\documentclass[12pt]{amsart}
\usepackage{mathrsfs}
\usepackage{txfonts}
\usepackage{amssymb}
\usepackage{color}
\makeatletter \@mparswitchfalse \makeatother
 \newtheorem{thm}{Theorem}[section]
 \newtheorem{cor}[thm]{Corollary}
 \newtheorem{lem}[thm]{Lemma}
 
 \theoremstyle{definition}
 \newtheorem{defn}[thm]{Definition}
 \theoremstyle{remark}

 \numberwithin{equation}{section}
 \usepackage{bbding}
\usepackage{mathrsfs}
\usepackage{txfonts}
\usepackage{amssymb}
\usepackage{color}

\begin{document}
%
%
%
%
%
%
%
%
%

\title[Noncommutative $H^p$ spaces]
 {Noncommutative $H^p$ spaces associated with type 1 subdiagonal algebras}

\author[Ruihan Zhang]{Ruihan Zhang}

\address{
School of Mathematics and Statistics,
  Shaanxi Normal University,
  Xian, 710119, People's  Republic of  China}

\email{ruihan@snnu.edu.cn}

\thanks{$^*$ Corresponding author\\
\indent This
research was supported by the National Natural
   Science Foundation of China(No. 11771261).}
\author[Guoxing Ji]{Guoxing Ji$^*$   } 
\address{School of Mathematics and Statistics,
  Shaanxi Normal University,
  Xian, 710119, People's  Republic of  China}
\email{gxji@snnu.edu.cn}



\begin{abstract}
Let $\mathfrak A$ be a type 1 subdiagonal algebra in a $\sigma$-finite von Neumann algebra $\mathcal M$ with respect to a faithful normal conditional  expectation $\Phi$.   We  consider  a Riesz type factorization theorem in noncommutative $H^p$ spaces associated with $\mathfrak A$. It is shown that if  $1\leq r,p,q<\infty$ such that $\frac1r=\frac1p+\frac1q$, then for any  $h\in H^r$, there exist  $h_p\in H^p$ and $h_q\in H^q$ such that $h=h_ph_q$.      Beurling type invariant subspace theorem    for noncommutative  $L^p(1< p<\infty)$ space is obtained. Furthermore, we show that a  $\sigma$-weakly closed subalgebra  containing  $\mathfrak A$ of  $\mathcal M$ is  also a  type 1 subdiagonal algebra. As an application, We prove that the relative  invariant subspace lattice $Lat_{\mathcal M}\mathfrak A$ of $\mathfrak A$ in $\mathcal M$ is commutative.
\end{abstract}

\maketitle
\baselineskip18pt
\section{Introduction}



 Riesz factorization theorem and Beurling invariant subspace theorem in the classical Hardy spaces are well-known.
  There were many very interesting extensions to abstract function algebras.    Arveson in \cite{arv1} introduced the notion of subdiagonal algebras as the noncommutative analogue of  the classical Hardy  space $H^{\infty}(\mathbb T)$. It is worth noting that there were several successful examples showing that some famous mature theorems about the classical Hardy spaces have been extended to the noncommutative Hardy spaces based on  subdiagonal algebras(cf. \cite{ble,ble1,ble2,ble3,jig,jig1,jig2,jio,jis}).
 Marsalli and West \cite{mar2} gave a  Riesz factorization theorem  for finite  noncommutative $H^p$ spaces.
  One important extension is due to Blecher and    Labuschagne on Beurling type invariant subspace  theorems for a  finite subdiagonal algebra in \cite{ble4}. Recently, Labuschagne in \cite{lab} extended  their results to noncommutative $H^2$ for  maximal subdiagonal algebras in a $\sigma$-finite von Neumann algebra based on noncommutative $H^p$ spaces in \cite{jig}. We also  discussed    Riesz type factorization theorem in noncommutative $H^1$  and  Beurling type invariant subspace theorem for $L^1(\mathcal M)$ associated with  a type 1 subdiagonal algebra in the sense that  every right invariant subspace of  a maximal subdiagonal algebra in  noncommutative  $H^2$ is of    Beurling type in \cite{jig2,jig3}. Do     Riesz factorization theorem and   Beurling type invariant subspace theorem when  $ 1<p<\infty $ hold  for     type 1 subdiagonal algebras? We consider these problems in this paper.
     We
   firstly
   recall some notions.

  Let $\mathcal M$ be
a $\sigma$-finite von Neumann algebra acting on a complex  Hilbert
$\mathcal H$. $Z(\mathcal M)=\mathcal M\cap\mathcal M^{\prime}$ is the center of $\mathcal M$, where $\mathcal M^{\prime}$ is the commutant of $\mathcal M$. If $Z(\mathcal M)=\mathbb C I$,
the multiples of identity $I$, then $\mathcal M$ is said  to be a factor.  We denote by $\mathcal M_*$ and $\mathcal M_p$  the space of all
$\sigma$-weakly continuous linear functionals of $\mathcal M$ and the set of all projections in $\mathcal M$ respectively. Let
$\Phi$  be  a faithful normal conditional expectation from $\mathcal
M$ onto  a von Neumann subalgebra   $\mathfrak D$. Arveson \cite{arv1} gave the following definition. A subalgebra
$\mathfrak A$ of $\mathcal M$, containing $\mathfrak D$, is called a
subdiagonal algebra of $\mathcal M$ with respect to $\Phi$ if

(i) $\mathfrak A \cap \mathfrak A^{*} =\mathfrak D$,

(ii) $\Phi$ is multiplicative on $\mathfrak A$, and

(iii)  $\mathfrak A + \mathfrak A^*$ is $\sigma$-weakly dense in
$\mathcal M$.

\noindent The algebra $\mathfrak D$ is called the diagonal of
$\mathfrak A$.
  we  may assume that  subdiagonal
 algebras are  $\sigma$-weakly closed without loss generality(cf.\cite{arv1}).

  We say that
$\mathfrak A$ is a maximal subdiagonal algebra in $\mathcal M$ with
respect to $\Phi$ in case that $\mathfrak A$ is not properly
contained in any other subalgebra of $\mathcal M$ which is
subdiagonal with respect to $\Phi$. Put $\mathfrak A_0 =\{X \in
\mathfrak A: \Phi(X)=0 \}$ and  $\mathfrak A_m =\{X \in \mathcal M:
\Phi(AXB)=\Phi(BXA)=0,~ \forall A \in \mathfrak A,~ B \in \mathfrak
A_0 \}$.
 By   \cite[Theorem 2.2.1]{arv1}, we recall that $\mathfrak A_m $ is a maximal
 subdiagonal algebra of $\mathcal M$ with respect to $\Phi$ containing
 $\mathfrak A$.  $\mathfrak A$ is said to be finite if there is a faithful normal finite trace $\tau$ on $\mathcal M$ such that $\tau\circ\Phi=\tau$.  Finite subdiagonal algebras are  maximal subdiagonal(cf.\cite{exe}).

 We  next recall Haagerup's noncommutative $L^p$ spaces associated with a
 general von Neumann algebra $\mathcal M$(cf.\cite{haa,ter}). Let $\varphi$ be a faithful
 normal state on $\mathcal M$ and  let $\{\sigma_t^{\varphi}:t\in\mathbb R\}$
  be the modular automorphism group of $\mathcal M$ associated with $\varphi$ by Tomita-Takesaki theory.
   We consider  the crossed product
 $\mathcal N=\mathcal M\rtimes_{\sigma^{\varphi}} \mathbb R$ of $\mathcal M$
 by $\mathbb R$ with respect to $\sigma^{\varphi}$.
   We denote by $\theta$ the dual action of $\mathbb R$ on $\mathcal
  N$. Then $\{\theta_s:s\in \mathbb R\}$  is an  automorphisms group
  of $\mathcal N$.

Note that $\mathcal M=\{X\in \mathcal N: \theta_s(X)=X,\forall
s\in\mathbb R\}$.  $\mathcal N$ is a semifinite von Neumann algebra
and there is the normal faithful semifinite trace $\tau$ on
$\mathcal N$ satisfying
$$\tau\circ \theta_s=e^{-s}\tau,  \ \ \  \forall s\in\mathbb R.$$

According to Haagerup \cite{haa,ter}, the noncommutative $L^p$
space  $L^p(\mathcal M)$ for each $0< p\leq \infty$ is defined as
the set of all $\tau$-measurable  operators $x$   affiliated with
$\mathcal N$ satisfying
$$
\theta_s(x)=e^{-\frac{s}{p}}x,  \ \forall  s\in\mathbb R.$$ There is
a linear bijection  between the predual $\mathcal M_*$ of $\mathcal
M$ and $L^1(\mathcal M)$: $f\to h_f$. If we define
$\mbox{tr}(h_f)=f(I), f\in$$ \mathcal M_*$, then $$
\mbox{tr}(|h_f|)=\mbox{tr}(h_{|f|})=|f|(I)=\|f\|$$ for all
 $f\in \mathcal M_*$ and
$$|\mbox{tr}(x)|\leq\mbox{tr}(|x|)$$ for all  $x\in L^1(\mathcal M)$.
As in \cite{haa}, we define the operator $L_A$ and $R_A$ on
$L^p(\mathcal M)$($1\leq p<\infty)$  by $L_Ax=Ax$ and $R_Ax=xA$ for
all  $A\in\mathcal M$ and $x\in L^p(\mathcal M)$.
 Note that
$L^2(\mathcal M)$ is a Hilbert space with the inner product $(a,b)
=\mbox{tr}(b^*a)$, $ \forall a,b \in L^2(\mathcal M)$ and
  $A\to L_A$( resp. $A\to R_A$) is a
  faithful representation (resp. anti-representation) of $\mathcal
  M$ on $L^2(\mathcal M)$. We may identify $\mathcal M$ with
  $L(\mathcal M)=\{L_A:A\in \mathcal M\}$.

  We recall that
  $\{\sigma_{t }^{\varphi}: t \in \mathbb R\}$ is the  modular automorphism
group of $\mathcal M$ associated with $\varphi$. Let $h_0$ be the
noncommutative Radon-Nikodym  derivative of the dual weight of
$\varphi$ on $\mathcal N$ with respect to $\tau$.  Then $h_0$ is the
image($h_{\varphi}$) of $\varphi$ in $L^1(\mathcal M)$ and we have
that the following representation of
$\sigma_t^{\varphi}$(cf.\cite{kos1,kos2}):
\begin{equation}\sigma_t^{\varphi}(X)=h_0^{it}Xh_0^{-it},\ \forall t\in\mathbb
R, \ \forall X\in\mathcal M .
   \end{equation}
It is known that the noncommutative $H^p$ space  $H^p(\mathcal M)$  and $H_0^p(\mathcal M)$ in $L^p(\mathcal M)$
for any $1\leq p<\infty$  is  $H^p=H^p(\mathcal M)=[h_0^{\frac{\theta}{p}}\mathfrak Ah_0^{\frac{1-\theta}p}]_p$ and $H_0^p=H_0^p(\mathcal M)=[h_0^{\frac{\theta}{p}}\mathfrak A_0h_0^{\frac{1-\theta}p}]_p$ for any $\theta\in[0,1]$
   from \cite[Definition
2.6]{jig} and \cite[Proposition 2.1]{jig1}. If $p=\infty$, then we  identify  $H^{\infty}$ as $\mathfrak A$ and $H_0^{\infty}$ as $\mathfrak A_0$.

We introduce the  notion of column $L^p$-sum in the  noncommutative $L^p$-space $L^p(\mathcal M)(1\leq p\leq \infty)$ for a $\sigma$-finite von Neumann algebra $\mathcal M$ given in \cite{js} by Junge and Sherman. For any subset $S\subseteq L^p(\mathcal M)$,   $[S]_p$ denotes  the closed linear span of  $S$ in $L^p(\mathcal M)$. If $X$ is a closed subspace of $L^p(\mathcal M)$, and
$\{X_i : i\in\Lambda \}$  is a family  of  closed subspaces
of $X$ such that  $X=[\bigvee\{X_i:i\in\Lambda\}]_p$ with the property that $X_j^*X_i=\{0\}$ if $i\not=j$,
then we say that X is the internal column $L^p$-sum  $\oplus_{i\in\Lambda}^{col}X_i$. If $p=\infty$, we  assume
that $X $ and $\{X_i:i\in\Lambda\}$ are $\sigma$-weakly closed, and the  closed linear span is taken with the $\sigma$-weak topology.  For symmetry,  if $X_jX_i^*=\{0\}$ if $i\not=j$ and  $X=[\vee\{X_i:i\in\Lambda\}]_p$, then we say that $X$ is the
  internal row $L^p$-sum  $\oplus_{i\in\Lambda}^{row}X_i$.

  Following \cite{ble4,lab}, we define the right wandering subspace of a right invariant subspace
$\mathfrak M\subseteq L^2(\mathcal M)$, that is, $\mathfrak M\mathfrak A\subseteq \mathfrak M$,   to be the space $W = \mathfrak M\ominus [\mathfrak M\mathfrak A_0]_2$.  We say that $\mathfrak M$ is of  type 1 if $W$ generates $\mathfrak M$ as an
$\mathfrak A$-module (that is, $\mathfrak M = [W\mathfrak A]_2$). We   say that $\mathfrak M$ is  of type 2 if $W = \{0\}$.
 Note that every   right invariant subspace $\mathfrak M$ is an $L^2$-column sum $\mathfrak M=\mathfrak N_1\oplus^{col}\mathfrak N_2$, where     $\mathfrak N_i$  is of  type $i$ for $i=1,2$ from \cite[Theorem 2.1]{ble4} and
\cite[Theorem 2.3]{lab}.  In particular,   if $\mathfrak M$ is of   type 1, then $\mathfrak M$ is of the Beurling type, that is, there exists a family  of  partial isometries $\{U_n:n\geq 1\}$ in $\mathcal M$ satisfying $U_n^*U_m=0$ if $n\not=m$ and $U_n^*U_n\in\mathfrak D$ such that $\mathfrak M=\oplus_{n}^{col}U_nH^2 $.
  We refer    to {\cite{ble4,lab} for  more details. Symmetrically,  we may consider left invariant subspaces for $\mathfrak A$ and a type 1 left invariant subspace $\mathfrak M$  may be represented as
 $\mathfrak M=\oplus_{n\geq 1}^{row}H^2V_n$ by a family of  partial isometries $\{V_n:n\geq1\}$ satisfying $V_nV_m^*=0$ if $n\not=m$ and $V_nV_n^*\in\mathfrak D$ if $n=m$ and this fact will be used frequently.
  Motivated by the column(resp. row)-sum as above, we call that   a family of partial isometries $\mathcal U=\{U_n:n\geq 1\}$ in $\mathcal M$ is column(resp. row) orthogonal if
  If $U_n^*U_m=0$(resp. $U_nU_m^*=0$) for $n\not=m$ and  $U_n^*U_n\in\mathfrak D$(resp. $U_nU_n^*\in\mathfrak D$) for all $n$(cf.\cite{jig3}).

In this paper, we continuously consider  type 1 subdiagonal algebras introduced in \cite{jig2}. Let $\mathfrak A$ be a type 1 subdiagonal algebra with respect  to $\Phi$ with diagonal $\mathfrak D$, that is, every invariant right invariant subspace $\mathfrak M\subseteq H^2$ of $\mathfrak A$ is of type 1. Note that there is a family of column orthogonal partial  isometries
$\mathcal U=\{U_n:n\geq 1\}$ such that $H_0^2=\oplus_{n}^{col}U_nH^2$ and those partial isometries together with the diagonal $\mathfrak D$ consist of generators of $\mathfrak A$(cf.\cite{jig2}). We will fixed those partial isometries throughout this paper.
 We give a Riesz type factorization theorem for non-commutative $H^P(1\leq p<\infty)$ which says that every element in $H^r$ is a product of two elements in $H^p$ and $H^q$ respectively with $\frac1r=\frac1p+\frac1q$ for any $1\leq r,p,q<\infty$ in Section 2. This gives an  answer  of a problem in \cite{jig2}.  Moreover,  we consider  a  Beurling type invariant subspace theorem  in  $L^p(\mathcal M)$ when $1<p<\infty$ in Section 3. In Section 4,  We determine  the structure of those  $\sigma$-weakly closed subalgebras containing  the  type 1 subdiagonal algebra $\mathfrak A$ in $\mathcal M$.  It is shown that those subalgebras are  of type 1. As an  application, we show that  the relative invariant subspace lattice $Lat_{\mathcal M}\mathfrak A=\{E\in\mathcal M_p:(I-E)AE=0,\ \forall A\in\mathfrak A\}$ of   $\mathfrak A$ in the von Neumann algebra $\mathcal M$  is  commutative in Section 5.

\section{Riesz   factorization in noncommutative $H^p$  spaces for type 1 subdiagonal algebras}

  In \cite{jig2}, we consider the Riesz factorization   in noncommutative $H^1$ for type 1 subdiagonal algebra $\mathfrak A$. It is shown that for any $h\in H^1$, there are two elements $h_i\in H^2(i=1,2)$ such that $ h=h_1h_2$.   Let $1\leq r,p,q< \infty$ such that  $\frac1r=\frac1p +\frac1q$ and $h\in H^r$. Does $h=h_ph_q$ for some $h_p\in H^p$ and $h_q\in H^q$?  We    consider this problem in this section.   For any subset $\mathcal  S\subseteq L^p(\mathcal M)$ and $ \mathcal K\subseteq L^q(\mathcal M)$, denote by $\mathcal S\mathcal K=[\{sk: s\in\mathcal S, k\in\mathcal K\}]_r$, the closed subspace  generated by  the subset $\{sk: s\in\mathcal S, k\in\mathcal K\}$   in $L^r(\mathcal M)$. We firstly recall that an element $h\in H^p$ is right(resp. left) outer in $[h\mathfrak A]_p=H^p$(resp. $[\mathfrak Ah]_p=H^p$).

\begin{lem} Let $\mathfrak A$ be a type 1 subdiagonal algebra. Then for any nonzero  $f\in L^2(\mathcal M)$, there exist a left outer element
$g\in H^2$ and a contraction $B\in \mathcal M$ such that $f=gB$.\end{lem}

\begin{proof}
We consider the type 1 subdiagonal algebra $\mathfrak A^*$ and  $(H^2)^*$,. Then for any  $f^*\in L^2(\mathcal M)$, there exist a right outer  element $ g^*\in (H^2)^*$
 and a contraction $B\in\mathcal M$ such that $f^*=B^*g^*$ by \cite[Theorem 3.5]{jig2}. Then $f=gB$. Note that $[g^*\mathfrak A^*]_2=(H^2)^*$. Thus
 $[\mathfrak Ag]_2=H^2$, that is,  $g\in H^2$ is left outer. \end{proof}

Let $\mathfrak M\subseteq L^p(\mathcal M)$ be a closed subspace.  If $\mathfrak M\mathfrak A \subseteq \mathfrak M$(resp. $\mathfrak A\mathfrak M\subseteq \mathfrak M$), then we say that $\mathfrak M$ is a right(resp. left) invariant subspace.
\begin{lem} Let $\mathfrak A$ be a type 1 subdiagonal algebra and $2<p<\infty$. If $\mathfrak M\subseteq H^p$ is a right invariant subspace such that $\mathfrak M H^s=H^2$, then $\mathfrak M=H^p$, where $\frac1p+\frac1s=\frac12$.\end{lem}

\begin{proof}
Let  $\frac1p+\frac1q=1$ and $\frac1s+\frac1t=1$. Then $\frac12+\frac1s=\frac1q$ and $\frac1p+\frac12=\frac1t$.
Put $\mathfrak M^{\bot}=\{g\in L^q(\mathcal M):\mbox{tr}(gm)=0,\forall m\in\mathfrak M\}$. It is sufficient to show that $\mathfrak M^{\bot}=(H^p)^{\bot}=H_0^q$. It is trivial that $\mathfrak M^{\bot}\supseteq H_0^q.$ Take any $g\in \mathfrak M^{\bot}$. Let $g=V|g|$ be the polar decomposition. Then
$g=V|g|^{\frac{q}{2}}|g|^{\frac{q}{s}}=g_1g_2$, where $g_1=V|g|^{\frac{q}{2}}\in L^2(\mathcal M)$ and $g_2=|g|^{\frac{q}{s}}\in L^s(\mathcal M)$. Then $g_1=h_1B$ for a left outer element $h_1\in H^2$ and a contraction $B\in \mathcal M$ by Lemma 2.1. Put $h_2=Bg_2\in L^s(\mathcal M)$. Now $\mbox{tr}(mg)=\mbox{tr}(mg_1g_2)=\mbox{tr}(mh_1h_2)=0$ for all $m\in\mathfrak M$.
On the other hand, $[\mathfrak Mh_1]_t=[\mathfrak M \mathfrak A h_1]_t=[\mathfrak M H_2]_t=[\mathfrak Mh_0^{\frac12}\mathfrak A]_t=[\mathfrak Mh_0^{\frac1s}h_0^{\frac1p}\mathfrak A]_t=[[\mathfrak Mh_0^{\frac1s}\mathfrak A]_2H^p]_t=[H^2H^p]_t=H^t$.
Thus $\mbox{tr}(xh_2)=0$ for all $x\in H^t$. It follows that $h_2\in (H^t)^{\bot}=H_0^s$. Therefore
$g=h_1h_2\in H_0^q$. Consequently, $\mathfrak M=H^p$.
\end{proof}

 \begin{thm} Let $\mathfrak A$ be a type 1 subdiagonal algebra and  $\eta \in L^p(\mathcal M)(1\leq p<\infty)$  a nonzero vector. Then for any $\delta>0$, there exist a contraction $B\in\mathcal M$ and a right outer element  $h\in H^p$  with $\|h\|_p<\|\eta \|_p+\delta $ such that $\eta =Bh$. If $\eta\in H^p$, then we may have $B\in\mathfrak A$.\end{thm}

 \begin{proof}   If $p=2$, then the  result is proved in \cite[Theorem 3.5]{jig2}. Suppose that   $2<p<\infty$ and we take some $t>2$ such that  $\frac12=\frac1p+\frac1t$.   Then $\eta h_0^{\frac1t}\in L^2(\mathcal M)$.

 Put $x=(\eta^*\eta +\varepsilon^2h_0^{\frac2p})^{\frac12}h_0^{\frac1t}$. Then we have that $x\in  L^2(\mathcal M)$ such that $\varepsilon^2h_0\leq x^*x $ and  $h_0^{\frac1t}\eta^*\eta h_0^{\frac1t}\leq x^*x $. By the proof of \cite[Theorem 3.5]{jig2},
      there are injective contractions $A, C\in\mathcal M$ with dense ranges, a unitary operator $U\in\mathcal M$ and  a right outer element  $\xi\in H^2$, such that $\varepsilon^{-1}Ax=h_0^{\frac12}$, $\eta h_0^{\frac1t}=Cx$ and  $x=U\xi$.
     Note that  $\xi=U^*x=U^*(\eta^*\eta +\varepsilon^2h_0^{\frac2p})^{\frac12}h_0^{\frac1t}$ is a right outer  element in $H^2$.
      Then  $H^2=[\xi\mathfrak A]_2=[U^*(\eta^*\eta +\varepsilon^2h_0^{\frac2p})^{\frac12}h_0^{\frac1t}\mathfrak A]_2=
      [U^*(\eta^*\eta +\varepsilon^2h_0^{\frac2p})^{\frac12}H^r]_2$. Put $ h=U^*(\eta^*\eta +\varepsilon^2h_0^{\frac2p})^{\frac12}$. We note that  $h\in H^p$. In fact, $H^2=[hh_0^{\frac1t}\mathfrak A]_2=[h\mathfrak A]_p H^t$   and $H_0^q=H^tH_0^2$. Thus $\mbox{tr}(h xy)=0$ for all $x\in H^t$ and $y\in H_0^2$, which implies that $h\in (H_0^q)^{\bot}=H^p$.
                Therefore $\mathfrak M=[h\mathfrak A]_p\subseteq H^p$ and  $\mathfrak MH^t=[\mathfrak Mh_0^{\frac1t}]_2=[h\mathfrak Ah_0^{\frac1t}]_2=[\xi\mathfrak A]_2=H^2$. By Lemma 2.2, $\mathfrak M=H^p$ and $h$ is right outer. Note that
        $\eta h_0^{\frac1t}=CU^*(\eta^*\eta +\varepsilon^2h_0^{\frac2p})^{\frac12}h_0^{\frac1t}=Chh_0^{\frac1t}$.   We then  have  $\eta=Ch$.  Note that $\|h\|_p= \|U^*(\eta^*\eta +\varepsilon^2h_0^{\frac2p})^{\frac12}\|_p=\left(\mbox{tr}((\eta^*\eta +\varepsilon^2h_0^{\frac2p})^{\frac{p}{2}})\right)^{\frac1p}\to \|\eta\|_p(\varepsilon\to0)$. We may choose  $\varepsilon>0$ so that  $\|h\|_p<\|\eta\|_p+\delta$.

        If $1\leq p<2$, then $\frac1p=\frac12+\frac1s$ for  some $s\geq2$. Let $\eta=V|\eta|=V|\eta|^{\frac{p}{2}}|\eta|^{\frac{p}{s}}=\eta_2\eta_s $, where $\eta_2=V|\eta|^{\frac{p}{2}}\in L^2(\mathcal M)$,
         $\eta_s=|\eta|^{\frac{p}{s}}\in L^s(\mathcal M)$ with   $\|\eta_2\|_2=\|\eta\|_p^{\frac{p}{2}}$ and $\|\eta_s\|_s=\|\eta\|_p^{\frac{p}{s}}$.
          We choose two positive numbers  $\delta_2,\delta_s$  such that $( \|\eta_2\|_2+\delta_2)(\|\eta_s\|_s+\delta_s)<\|\eta_2\|_2\|\eta_s\|_s+\delta=\|\eta\|_p+\delta$.
          By \cite[Theorem 3.5]{jig2} and the proof as above,  $\eta_s=Ch_s$ for a right outer
         element $h_s\in H^s$ with $\|h_s\|_s<\|\eta_s\|_s+\delta _s $ and a contraction $C\in\mathcal M$. Now $\eta_2C\in L^2(\mathcal M)$. We again have $\eta_2C=Bh_2$ for a right outer
         element $h_2\in H^2$  with $\|h_2\|_2<\|\eta_2\|_2+\delta_2$ and a contraction $B\in\mathcal M$. It is elementary that
         $h=h_2h_s\in H^p$ is right outer such  that $\|h\|_p\leq \|h_2\|_2\|h_s\|_s<\|\eta_2\|_2\|\eta_s\|_s+\delta=\|\eta\|_p+\delta$.

 If  $\eta \in H^p$, then $Bh A=\eta A\in H^p$ for all $A\in \mathfrak A$. Thus $B H^p\subseteq H^p$ and  $B\in \mathfrak A$  from \cite[Theorem 2.7]{jig}.
  \end{proof}

 We now have the  following Riesz type factorization    theorem in noncommutative $H^p$ spaces for $1\leq p<\infty$ whose proof  is similar to \cite[Theorem 3.6]{jig2}. To completeness, we give the detail.

  \begin{thm}Let $1\leq r,p,q<\infty$ such that  $\frac1r=\frac1p+\frac1q$ and let $\mathfrak A$ be a type 1 subdiagonal algebra.   Then for any $h\in H^r$ and  $\varepsilon>0$,
  there exist $h_p\in H^p$, $h_q\in H^q$  with $\|h_t\|_t<\|h\|_r^{\frac{r}{t}}+\varepsilon$  for $t=p,q$   such that $h=h_ph_q$.
  If $h\in H_0^r$, we may choose  one of $h_t$ in $H_0^t$ for $t=p,q$. \end{thm}

   \begin{proof} Take  $1<s\leq \infty$  such that  $\frac1r +\frac1s=1$.
   Let $h=V|h|=V|h|^{\frac{r}p}|h|^{\frac{r}q}$ be the polar decomposition. Then $|h|^{\frac{r}t}\in  L^t(\mathcal M)$   with $\||h|^{\frac{r}t}\|_t=\|h\|_r^{\frac{r}t}$  for $t=p,q$.
    By Theorem 2.3, there exist a contraction $B\in \mathcal M$ and a right  outer element $h_q\in H^q$  with $\|h_q\|_q<\||h|^{\frac{r}{q}}\|_q+\varepsilon=\|h\|_r^{\frac{r}{q}}+\varepsilon$ such that $|h|^{\frac{r}q}=Bh_q$. Put $h_p=V|h|^{\frac{r}{p}}B$. We show that $h_p\in H^p$.
   In fact, $tr(hg)=0$ for all $g\in H^s_0$ since $h\in H^r$. Then
   $tr(h g)=tr(h_ph_qg)=0$ for all  $g\in H^s_0$.  Note that $h_q$ is right outer. Then $[h_qH^s_0]_{p^{\prime}}=[h_q\mathfrak A H^s_0]_{p^{\prime}}=H^qH^{s}_0=H^{p^{\prime}}_0$, where $p^{\prime}>0$ such that  $\frac1p+\frac1{p^{\prime}}=1$.
    Therefore  $h_p\in H^p$ with    $\|h_p\|_p\leq \||h|^{\frac{r}p}\|_p< \|h\|_r^{\frac{r}p}+\varepsilon$. Moreover, if $h\in H_0^p$, then $tr(hg)=0$ for all $g\in H^s$. It easily follows that
   $h_p\in H_0^p$  in this case.

   On the other hand,  by taking   the pair $(h^*, (H^r)^*)$ if $h\in (H_0^r)$, we may write $h^*$ in the form $h^*=g_qg_p$ with $g_q\in (H_0^q)^*$ and $g_p\in (H^p)^*$. Then   we have $h=h_ph_q$ with $h_p\in H^p$ and  $h_q\in H_0^q$   by setting $h_q=g_q^*$ and $h_p=g_p^*$. \end{proof}

\section{Beurling type invariant subspace theorem for type 1 subdiagonal algebras}

Let $\mathfrak A$ be  a type 1 subdiagonal algebra.
For any  positive integer $n\geq 1$, let  $\mathfrak A_0^n$ be the  $\sigma$-weakly closed ideal of $\mathfrak A$  generated by $\{a_1a_2\cdots a_n:a_j\in\mathfrak A_0\}$.  It is elementary that a right invariant  subspace $\mathfrak M\subseteq L^2(\mathcal M)$ is of type 1(resp. type 2) if and only if $\cap_{n\geq 1}[\mathfrak M\mathfrak A_0^n]_2=\{0\}$(resp. $\mathfrak M=[\mathfrak M\mathfrak A_0]_2$). Based on this fact, we introduce similar notions in $L^p$ spaces.
 \begin{defn} Let $\mathfrak A$ be  a type 1 subdiagonal algebra and  $\mathfrak M$  a right(resp. left) invariant subspace in $L^p(\mathcal M)$. Then we say that $\mathfrak M$  is of type 1 if $\cap_{n\geq 1}[\mathfrak M\mathfrak A_0^n]_p=\{0\}$(resp. $\cap_{n\geq 1}[\mathfrak A_0^n\mathfrak M]_p=\{0\}$) and of type 2 if $\mathfrak M=[\mathfrak M\mathfrak A_0]_p$(resp. $\mathfrak M=[\mathfrak A_0\mathfrak M]_p$).
 \end{defn}
     We recall that there exists a column orthogonal family of  partial isometries $\{U_n:n\geq1\}$ in $\mathcal M$  such that   \begin{equation}
 H_0^2 = \oplus_{n\geq 1}^{col} U_n  H^2
 \end{equation} when $\mathfrak A$ is of type 1.

 We note that a Beurling  type invariant subspace theorem  is given in \cite{jig3} when  $p=1$ or $p=\infty$.
We now  consider this theorem  in noncommutative $L^p(\mathcal M)$ for $1< p< \infty$.
\begin{lem} Let $\mathfrak A$ be a type 1 subdiagonal algebra.  If   $\mathfrak M \subseteq  L^p(\mathcal M)$ ($1<p<\infty$) is a right invariant subspace of type 2, then there exists a projection $E
\in\mathcal M$ such that $\mathfrak M=E L^p(\mathcal M)$.\end{lem}
\begin{proof}
By \cite[Theorem 2.7]{jig2},
$\mathfrak A_0=\vee\{\mathfrak DU_{i_1}U_{i_2}\cdots U_{i_n}:i_j\geq 1\}$, where $\{U_i:i\geq 1\}$ is a family of column orthogonal partial isometries as in $(3.1)$.
 Note that $\mathfrak M$ is  $\mathfrak D$ right  invariant  and   $[\mathfrak M\mathfrak A_0]_p=\mathfrak M$.  Since
 $\vee\{\mathfrak MU_n:n\geq1\}\supseteq \vee\{\mathfrak MU_{i_1}U_{i_2}\cdots U_{i_n}:i_k\geq 1,n\geq1\}$, $\mathfrak M=\vee\{\mathfrak M U_n:n\geq1\}$.  For any $m,n \geq 1$ and $x\in \mathfrak M$, we have $R_{U_m^*}R_{U_n}x=xU_nU_m^*\in \mathfrak M$ since $U_nU_m^*\in\mathfrak D$ from \cite[Proposition 2.3]{jig2}.   Then $\mathfrak M$ is right $\mathcal M$ invariant. That is, $\mathfrak M$ is an $\mathcal M$ right module.  Then  by  \cite[Theorem 3.6]{js},
 there is a projection $E\in\mathcal M$ such that $\mathfrak M=EL^p(\mathcal M)$. In fact, we also have that the right submodule  $\mathfrak M$ is column summand in  $L^p(\mathcal M)$ from the proof of this theorem(cf.\cite[$p_{19}$]{js}). It follows that  the projection onto $\mathfrak M$ commutes with the right multiplications  by  $\mathcal M$. Thus there is  a projection $E\in\mathcal M$ such that $\mathfrak M=EL^p(\mathcal M)$ by \cite[Corollary 2.6]{js}.
\end{proof}

 For a family of column orthogonal partial isometries
  $\mathcal W=\{W_n:n\geq 1\}$, we define a right invariant subspace
  $\mathfrak M_{\mathcal W}^p=\oplus_{n\geq 1}^{col}W_n H^p$ in $ L^p(\mathcal M)$ for any $1\leq p\leq \infty$  as in \cite{jig3}.
If $1\leq r<\infty$ and $\frac1r=\frac1p+\frac1q$ for some  $p, q\geq 1$, then
    $\mathfrak M_{\mathcal W}^r=[(\oplus_{n\geq 1}W_n H^p) H^q]_r=\mathfrak M_{\mathcal W}^p H^q$.  The following theorem is motivated from \cite[Theorem 2.2]{jig3}.
\begin{thm}
Let   $\mathfrak A$ be a type 1 subdiagonal algebra and let  $\mathfrak M\subseteq  L^p(\mathcal M)$$(1< p< \infty)$ be a nonzero closed right invariant subspace of $\mathfrak A$.

$(i)$  There  are   right invariant subspaces $\mathfrak M_i$ of type $i$ $ (i=1,2)$ such that $\mathfrak M=\mathfrak M_1\oplus^{col}\mathfrak M_2$.

$(ii)$  If  $\mathfrak M\subseteq L^p(\mathcal M)$ is a right invariant subspace of type 1, then there exists a family of column orthogonal partial isometries $\mathcal W$ such that $\mathfrak M=\mathfrak M_{\mathcal W}^p$.
\end{thm}
\begin{proof}
$(i)$ Put $\mathfrak M_2=\cap_{n\geq 1}[\mathfrak M\mathfrak A_0^n]_p$. Then $\mathfrak M_2\subseteq \mathfrak M$ is a right invariant subspace of type 2. By Lemma  3.1,
 there is a projection $E\in\mathcal M$, such that $\mathfrak M_2=EL^p(\mathcal M)\subseteq \mathfrak M$. Put $\mathfrak M_1=(I-E)\mathfrak M\subseteq \mathfrak M$. This mean that $\mathfrak M_1$ is  closed and right invariant
 such that $\mathfrak M=\mathfrak M_1\oplus^{col}\mathfrak M_2$. Note that  $\cap_{n\geq 1}[\mathfrak M_1\mathfrak A_0^n]_p\subseteq \cap_{n\geq 1}[\mathfrak M\mathfrak A_0^n]_p\subseteq\mathfrak M_1\cap\mathfrak M_2=\{0\}$. Then  $\mathfrak M_1$ is of type 1.

$(ii)$ Firstly, we assume that $1< p<2$ and take $r>2$  with $\frac1p=\frac12+\frac1r$.
Put \begin{align*}\mathcal P=&  \{\mathcal W=\{W_n\}_{n\geq 1}\subseteq\mathcal M:
 \mbox{column orthogonal }\\ &\mbox{ partial isometries such that }
  \mathfrak M_{\mathcal W}^2 H^r\subseteq \mathfrak M  \}.\end{align*}
 Note that $\mathcal P$ is non-empty.
 In fact, we show that  for  any nonzero  $h\in \mathfrak M$, $[h\mathfrak A]_p=(VH^2)H^r$ for  a partial isometry $V\in\mathcal M$  with $V^*V\in \mathfrak D$. Write $h=(u|h|^{\frac{p}2})|h|^{\frac{p}r}$. Then  $|h|^{\frac{p}r}\in L^r(\mathcal M)$ by
  Theorem 2.3,    there exist a contraction $B\in \mathcal M$ and  a right  outer element $h_r\in  H^r$ such that  $|h|^{\frac{p}r}=Bh_r$.  Put $h_2=u|h|^{\frac{p}2}B$. We have $h=h_2h_r$. This implies that
   $$\mathfrak M\supseteq [h\mathfrak A]_p=[h_2h_r\mathfrak A]_p=[h_2H^r]_p=[h_2\mathfrak Ah_0^{\frac1r}\mathfrak A]_p=[h_2\mathfrak A]_2 H^r$$ since $[h_r\mathfrak A]_r=H^r$.
    Note that $[h_2\mathfrak A]_2\subseteq  L^2(\mathcal M)$ is a right invariant subspace. By \cite[Lemma 3.3]{jig2},
      $[h_2\mathfrak A]_2=V H^2\oplus^{col} N_2$ for a partial isometry $V\in\mathcal M$ such that $V^*V\in\mathfrak D$
      and a  right invariant subspace $N_2$ of type 2.  Note that
      $$[ N_2H^r]_p=[[ N_2\mathfrak A_0]_2H^r]_p=[[ N_2H^r]_p\mathfrak A_0]_p.$$
      It follows that $[N_2H^r]_p=\cap_{n\geq 1}[(N_2H^r)\mathfrak A_0^n]_p\subseteq \mathfrak M=\{0\}$.
      Thus $N_2=\{0\}$ and $[h_2\mathfrak A]_2=VH^2$ with $V^*V\in \mathfrak D$.

       Note that $h_r$ is right outer. This means that  $[h\mathfrak A]_p=[h_2h_r\mathfrak A]_p=[h_2H^r]_p=[h_2\mathfrak AH^r]_p=[h_2\mathfrak A]_2H^r= (VH^2)H^r$ and  $\{V\}\in\mathcal P$.

 We define a partial order in $\mathcal P$ by $\mathcal W\leq \mathcal V$ if $\mathfrak M_{\mathcal W}^2\subseteq \mathfrak M_{\mathcal V}^2$ for any $\mathcal W, \mathcal V\in \mathcal P$. Let $\{\mathcal W_{\lambda}:\lambda \in\Lambda\}\subseteq \mathcal P$ be a totally ordered family in $\mathcal P$. Put $\mathfrak N=\vee\{\mathfrak M_{\mathcal W_{\lambda}}^2: \lambda\in\Lambda\}$. Note that $\mathfrak N\subseteq  L^2(\mathcal M)$ is a right invariant subspace in $L^2(\mathcal M)$. Then  $\mathfrak N=\mathfrak M_{\mathcal W}^2\oplus^c \mathfrak N_2$ for a family of column orthogonal partial isometries $\mathcal W$ and a right invariant subspace $\mathfrak N_2$ of type 2 in $L^2(\mathcal M)$.  Then $[\mathfrak N_2 H^r]_p$ is also a type 2 right invariant subspace in $ L^p(\mathcal M)$.
 This implies that $[\mathfrak N_2 H^r]_p=\cap_{n\geq 1} [(\mathfrak N_2 H^r)\mathfrak A_0^n]_p
  \subseteq \cap_{n\geq 1}[\mathfrak M\mathfrak A_0^n]_p=\{0\}$.   Hence $\mathfrak N_2=\{0\}$.
 Since $\mathfrak M_{\mathcal W_{\lambda}}^2 H^r\subseteq \mathfrak M$, $\mathfrak N H^r\subseteq \mathfrak M$. Thus
  $\mathcal W\in\mathcal P$ with $\mathcal W_{\lambda}\leq \mathcal W$ for any $\lambda\in \Lambda$. That is, $\mathcal W$ is an upper bound of $\{\mathcal W_{\lambda}:\lambda\in\Lambda\}$. By Zorn's lemma, there exists a maximal element $\mathcal W\in \mathcal P$. We show  that $\mathfrak M=\mathfrak M_{\mathcal W}^p$.
  Otherwise, assume that  there is an $h\in \mathfrak M$ such that $h\notin\mathfrak M_{\mathcal W}^p$.
   Then $[h\mathfrak A]_p=(VH^2)H^r$ for a partial isometry $V\in\mathcal M$ with $V^*V\in\mathfrak D$ as above.
     Since $h\notin \mathfrak M_{\mathcal W}^p$, $V H^2\nsubseteq\mathfrak M_{\mathcal W}^2$. Now
$\tilde{\mathfrak N}=\vee\{\mathfrak M_{\mathcal W}^2,V H^2\}$ is also a right invariant subspace in $ L^2(\mathcal M)$ with $\tilde{\mathfrak N} H^r\subseteq \mathfrak M$.  In this case, we have $\tilde{\mathfrak N}=\mathfrak M_{\mathcal V}^2$ for a family of  column orthogonal partial isometries $\mathcal V\in\mathcal P$. It is trivial  that $\mathfrak M_{\mathcal W}^2\subsetneqq\mathfrak M_{\mathcal V}^2$ by a similar treatment. This is a contradiction. Thus
$\mathfrak M=\mathfrak M_{\mathcal W}^2H^r=\mathfrak M_{\mathcal W}^p$.

  We have  that  for a general right invariant subspace $\mathfrak M\subseteq L^p(\mathcal M)$, there is a right invariant subspace $\mathfrak N\subseteq L^2(\mathcal M)$ such that $\mathfrak M=\mathfrak NH^r$.
  In fact,  by $(i)$, $\mathfrak M=\mathfrak M_1\oplus^{col}\mathfrak M_2=\mathfrak M_{\mathcal W}^2H^r\oplus^{col}EL^2(\mathcal M)H^r=
  (\mathfrak M_{\mathcal W}^2\oplus^{col} EL^2(\mathcal M))H^r=\mathfrak NH^r$, where $\mathfrak N=\mathfrak M_{\mathcal W}^2\oplus^{col}EL^2(\mathcal M)$.

 Let $2<p< \infty$ and let $p$ and $q$ be conjugate exponents, that is, $\frac1p+\frac1q=1$. Then $1< q<2$ and  $\frac12=\frac1p+\frac1r$ as well as $\frac1q=\frac12+\frac1r$ for some  $1<r<\infty$. Note that  $\mathfrak M H^r\subseteq L^2(\mathcal M)$ is right invariant of type 1.
 Thus there is a family of column orthogonal partial isometries $\mathcal W=\{W_n:n\geq 1\}$ such that
 $\mathfrak MH^r=\mathfrak M_{\mathcal W}^2=\mathfrak M_{\mathcal W}^pH^r$.

 We next show that $\mathfrak M=\mathfrak M_{\mathcal W}^p$. For any $x\in\mathfrak M$, it is elementary
 that $xh_0^{\frac1r}=\oplus^{col}_{n\geq 1}W_nW_n^*xh_0^{\frac1r}$ with $W_n^*xh_0^{\frac1r}\in H^2$  since $W_n^*xh_0^{\frac1r}\in H^p$. Then $W_n^*x\in H^p$ and
 $x=\oplus^{col}_{n\geq 1}W_nW_n^*x\in \mathfrak M_{\mathcal W}^p$. Therefore $\mathfrak M\subseteq \mathfrak M_{\mathcal W}^p$.

 For any closed subspace $K\subseteq L^p(\mathcal M)$, we put $K^{\bot}=\{y\in L^q(\mathcal M):tr(xy)=0,\forall x\in K\}$.
 It is known that $K^{\bot}$ is  left invariant when $K$ is right invariant. By symmetry, we easily have
 $\mathfrak M^{\bot}=H^r\mathfrak N=\oplus_{n\geq 1}^{row} H^qV_n\oplus ^{row} L^q(\mathcal M)F$  for a left invariant subspace $\mathfrak N=\oplus_{n\geq 1}^{row} H^2V_n\oplus ^{row} L^2(\mathcal M)F\subseteq L^2(\mathcal M)$ since $1< q<2$ by preceding proof.

  We claim that $(\mathfrak MH^r)^{\bot}=\mathfrak N$. It is clear that $(\mathfrak MH^r)^{\bot}\supseteq \mathfrak N$. Take any $y\in(\mathfrak MH^r)^{\bot}$.  Then $\mbox{tr}(xhy)=0$ for all $x\in \mathfrak M$ and  $h\in H^r$. It follows that   $H^ry\subseteq \mathfrak M^{\bot}=\oplus_{n\geq 1}^{row} H^qV_n\oplus ^{row} L^q(\mathcal M)F$.
 By a similar treatment as above, we have
 $h_0^{\frac1r}y=\oplus_{n\geq 1}^{row}h_0^{\frac1r}yV_n^*V_n \oplus^{row}h_0^{\frac1r}yF$ and thus $y=\oplus_{n\geq 1}^{row}yV_n^*V_n \oplus^{row}yF\in \mathfrak N$.

 Thus $\mathfrak M^{\bot}=H^r\mathfrak N=H^r(\mathfrak MH^r)^{\bot}=H^r(\mathfrak M_{\mathcal W}^pH^r)^{\bot}=(\mathfrak M_{\mathcal W}^p)^{\bot} $
 and $\mathfrak M=\mathfrak M_{\mathcal W}^p$.
\end{proof}

 Let $\mathfrak M\subseteq L^p(\mathcal M)$ be a right invariant subspace. Then $\mathfrak M=\mathfrak M_{\mathcal W}^p\oplus^{col}EL^p(\mathcal M)$ for a family of column orthogonal partial isometries $\mathcal W$  and a projection $E$ in $\mathcal M$. It follows that $\mathfrak M_{\infty}=\mathfrak M_{\mathcal W}^{\infty}\oplus E\mathcal M$ is a $\sigma$-weakly closed right invariant subspace in $\mathcal M$.

\begin{cor}
Let   $\mathfrak A$ be a type 1 subdiagonal algebra. Then there is a lattice isomorphism from the lattice of all  right invariant subspaces of $\mathfrak A$ in $\mathcal M$ onto that in $L^p(\mathcal M)$ for $1\leq p<\infty$.
\end{cor}
\begin{proof}
Let $\mathfrak M \subseteq\mathcal M$ be a $\sigma$-weakly closed right invariant subspace. Then $\mathfrak M_p=[\mathfrak Mh_0^{\frac1p}]_p=\mathfrak MH^p\subseteq L^p(\mathcal M)$ is a right invariant subspace. We show this correspondence is a bijection. Assume that $\mathfrak MH^p=\mathfrak NH^p$ for a $\sigma$-weakly closed right invariant subspace $\mathfrak N\subseteq \mathcal M$. Then we have two column orthogonal families of partial isometries $\mathcal W$ and $\mathcal V$  and two projections $E,F\in\mathcal M$ such that
$ \mathfrak M=\mathfrak M_{\mathcal W}^{\infty}\oplus ^{col}E \mathcal M $ and $ \mathfrak N=\mathfrak M_{\mathcal V}^{\infty}\oplus ^{col}F \mathcal M $ by \cite[Theorem 2.2]{jig3}. For any $A\in\mathfrak M$, $Ah_0^{\frac1p}\in\mathfrak MH^p=\mathfrak NH^p$. Then $Ah_0^{\frac1p}=\oplus_{n\geq 1} V_nV_n^*Ah_0^{\frac1p}\oplus FAh_0^{\frac1p}$ with $V^*_nA\in\mathfrak A$ and $FAh_0^{\frac1p}\in FL^p(\mathcal M)$. It follows that $A=\oplus V_nV_n^*A+FA\in\mathfrak N$. That is, $\mathfrak M\subseteq \mathfrak N$. By symmetry, we have $\mathfrak M=\mathfrak N$. On the other hand, for any right invariant subspace $\mathfrak M_p\subseteq  L^p(\mathcal M)$, there is a family of column orthogonal partial isometries $\mathcal W $ and a projection $E\in\mathcal M$ such that $\mathfrak M_p=\mathfrak M_{\mathcal W}^p\oplus EL^p(\mathcal M)$ from Theorem 2.3. Put $\mathfrak M=\mathfrak M_{\mathcal W}^{\infty}\oplus^{col}E\mathcal M$. Then $\mathfrak M_p=\mathfrak MH^p$. Thus the map: $\mathfrak M\to \mathfrak MH^p$ is a bijection from the lattice of all $\sigma-$weakly closed right invariant subspaces in $\mathcal M$  onto that of all  right closed invariant subspaces in $L^p(\mathcal M)$.
\end{proof}

\section{Subalgebras containing a type 1 subdiagonal algebra}

We gave a necessary and sufficient condition for a type 1 subdiagonal algebra $\mathfrak A$ to be  a maximal subalgebra of  $\mathcal M$ in \cite{jig3}.  In fact, we may determine all $\sigma-$weakly closed subalgebras
  containing $\mathfrak A$ in  $\mathcal M$. Let $\mathcal B$ be a $\sigma-$weakly closed subalgebra of $\mathcal M$ containing $\mathfrak A$. Then $[\mathfrak Bh_0^{\frac12}]_2$ is a right invariant subspace in $L^2(\mathcal M)$.  As in \cite{jig3}, let $p$ and $q$ be the projections  from $L^2(\mathcal M)$ onto $L^2(\mathfrak D)$ and the wandering subspace $[\mathfrak Bh_0^{\frac12}]_2\ominus [\mathfrak Bh_0^{\frac12}\mathfrak A_0]_2$ of $[\mathfrak Bh_0^{\frac12}]_2$ respectively. It is known that both $p$ and $q$ are in the commutant  $(R(\mathfrak D))^{\prime}$ of $R(\mathfrak D)$. For any two projections $E$ and $F$ in a von Neumann algebra $\mathcal M$, $E\preccurlyeq F$ means that there is a partial isometry $V\in\mathcal M$ such that $V^*V=E$ and $VV^*\leq F$. If $E\preccurlyeq F$ and $F\preccurlyeq E$, then  $E\sim F$. For an operator $A$,  $R(A)$ denotes the range of $A$.

\begin{lem} Let $\mathfrak A$ be  a type 1 subdiagonal algebra and $\mathfrak B$ as above. Then  $q\leq p$.
\end{lem}
\begin{proof}
 Note that $\mathfrak B$ is two-side $\mathfrak A$ invariant subspace. Then  there exists a family of column orthogonal  partial isometries $\mathcal W$ and a projection $E $ in $\mathcal M$ such that $\mathfrak B=\oplus_{n\geq 1}^{col}W_n \mathfrak A\oplus ^{col}E \mathcal M $ as well as  $[\mathfrak B h_{0}^{\frac12}]_2=\oplus_{n\geq 1}^{col}W_n H^2\oplus ^{col}E L^2(\mathcal M) $ by Corollary 3.1. This means that   $E L^2(\mathcal M)=\cap_{n\geq 1}[\mathfrak B h_{0}^{\frac12}\mathfrak A_0^n]_2$ . Since $[\mathfrak B h_{0}^{\frac12}\mathfrak A_0^n]_2$ is left $\mathfrak B$ invariant for any $n$,  $E L^2(\mathcal M)$ is also left $\mathfrak B$ invariant. This implies that $EBE=EB$  for any $B\in\mathfrak B$. In particular, $ED=DE$ for all $D\in \mathfrak B \cap \mathfrak B^*$. We now have $W_{n}\in \mathfrak B \cap \mathfrak B^*$  for any $n\geq 1$. In fact, it is trivial that $W_{n}\in \mathfrak B$ for any $n\geq1$. On the other hand, $W_{n}^*\mathfrak B=W_{n}^*W_{n}\mathfrak A \subseteq \mathfrak A \subseteq \mathfrak B$ since $W_{n}^*W_{n}\in \mathfrak D$. Then $W_{n}^*\in \mathfrak B$ for any $n\geq1$. Since $E W_{n}=0$ for any $n$, $E W_{n}=W_{n}E=0$ and $\oplus_{n\geq 1}W_nW_{n}^*=I-E$. Put $\mathfrak N=\oplus_{n\geq 1}^{col}W_n H^2\subseteq[\mathfrak B h_{0}^{\frac12}]_2$.  Then $\mathfrak N$ is the type 1 part of the  right invariant subspace  $[\mathfrak Bh_0^{\frac12}]_2$.

Take any subprojections $p_1\leq p$ and  $q_1 \leq q$    such that $p_1\sim q_1$  in $( R(\mathfrak D))'$. We begin by proving that $p_1\leq q$. Put $\mathfrak M_1=[R(p_1)\mathfrak A]$ and $\mathfrak N_1=[R(q_1)\mathfrak A]$. Let $w_1 \in ( R(\mathfrak D))'$ be a partial isometry such that $w_1^*w_1=q_1$ and $w_1w_1^*=p_1\leq p$ as in \cite[Proposition 2.5]{jig3}. Then $\mathfrak N_1=W^*H^2$ for a partial isometry $W\in \mathcal M$ such that $H^2=W\mathfrak N_1\oplus^{col} (I-WW^*)H^2$ by \cite[Proposition 2.5]{jig3} again. It follows that $\mathfrak N_1=W^*H^2=W^*W\mathfrak N_1$. Note that $\mathfrak N=[\left([\mathfrak B h_{0}^{\frac12}]_2\ominus[\mathfrak B h_{0}^{\frac12}\mathfrak A_0]_2\right)\mathfrak A]_2=[R(q)\mathfrak A]_2$. We have $\mathfrak N=[R(q_1)\mathfrak A]_2\oplus^{col}[(R(q-q_1))\mathfrak A]_2=\mathfrak N_1\oplus^{col}[(R(q-q_1))\mathfrak A]_2$. We similarly obtain that $H^2=[R(p_1)\mathfrak A]_2\oplus^{col}[(R(p-p_1))\mathfrak A]_2=\mathfrak M_1\oplus^{col}[(R(p-p_1))\mathfrak A]_2$. It follows that $W^*H^2=W^*\mathfrak M_1=\mathfrak N_1$. Then we have $\mathfrak M_1=W\mathfrak N_1$ since $W^*\mathfrak M_1=W^*H^2=W^*W\mathfrak N_1$. We also get $WW^*H^2=WW^*\mathfrak M_1=\mathfrak M_1\subseteq H^2$ and so that $WW^*\in \mathfrak D$. Note that $[\mathfrak B h_{0}^{\frac12}]_2=[R(q)\mathfrak A]_2\oplus ^{col}E L^2(\mathcal M) $. We have $[R(q)\mathcal M]_2=(I-E)L^2(\mathcal M)$ since $\oplus_{n\geq 1}^{col}W_nW_{n}^*=I-E$. It is elementary that $W$ is an isometry on $[R(q_1)\mathcal M]_2$ while  $Wx=0$ for any $x\in [R(q)\mathcal M]_2^\bot$. It follows that $W(E L^2(\mathcal M))=0$ and we in turn have $WE=0$. Thus $W[\mathfrak B h_{0}^{\frac11}]_2=W[R(q)\mathfrak A]_2=W[R(q_1)\mathfrak A]_2=W\mathfrak N_1=\mathfrak M_1\subseteq H^2\subseteq [\mathfrak B h_{0}^{\frac12}]_2$. We now have $W\in \mathfrak B$ by \cite[Proposition 2.3]{jig3}.

 On the other hand, $W^*H^2=W^*\mathfrak M_1=\mathfrak N_1\subseteq \mathfrak N=\oplus_{n\geq 1}^{col}W_nH^2$. Then    $W^*h_{0}^{\frac12}=\oplus_{n\geq 1}^{col}W_n\xi_n$  for some  $\xi_n \in H^2$. It follows that $W_n^*W^*h_{0}^{\frac12}=W_n^*W_n\xi_n\in H^2$, which implies that $W_n^*W^*\in \mathfrak A\subseteq \mathfrak B$. Moreover, $W_nW_n^*W^*h_{0}^{\frac12}=W_n\xi_n$. So we get $W^*h_{0}^{\frac12}=\oplus_{n\geq 1}^{col}W_n\xi_n=\oplus_{n\geq 1}^{col}W_nW_n^*W^*h_{0}^{\frac12}$. Then $W^*=\oplus_{n\geq 1}^{col}W_nW_n^*W^*$ since $h_{0}^{\frac12}$ is separating. We now have $W^*\in \mathfrak B$. It follows that $\mathfrak M_1=WW^*\mathfrak M_1=WW^*H^2\subseteq WW^*[\mathfrak B h_{0}^{\frac12}]\subseteq W[\mathfrak B h_{0}^{\frac12}]_2=W\mathfrak N_1=\mathfrak M_1$. Thus $WW^*H^2= WW^*[\mathfrak B h_{0}^{\frac12}]_2=\mathfrak M_1$  with the  wandering subspace $R(p_1)$.  Take any $\xi\in R(p_1)$.   We have
  $$\langle \xi,B h_{0}^{\frac12} A\rangle=\langle WW^*\xi,B h_{0}^{\frac12} A\rangle =\langle \xi,WW^*B h_{0}^{\frac12}A\rangle=0$$ for any $ B\in \mathfrak B$ and $A\in\mathfrak A_0$. Moreover, we have $\xi\in R(p_1)\subseteq R(p)=L^2(\mathfrak D)\subseteq H^2\subseteq [\mathfrak B h_{0}^{\frac12}]_2$.
  It follows that $R(p_1) \subseteq [\mathfrak Bh_0^{\frac12}]_2  \ominus [\mathfrak Bh_0^{\frac12}\mathfrak A_0]_2=R(q)$.
         Therefore
          $p_1\leq q$.

We take a maximal family of mutually orthogonal subprojections $\{p_n:n\geq1\}$ of $p$ such that  $p_n\sim q_n$ for  some subprojections $q_n\leq q$ for all $n\geq1$  in $(R(\mathfrak D))^{\prime}$ by Zorn's lemma. Then we must have $p_n\leq q$ for all $n\geq 1$ as above. Put $p_0=p-(\oplus_{n\geq 1}p_n)$ and $q_0=q-(\oplus_{n\geq 1}p_n)$, then $p_0$ and $q_0$ are in $(\mathcal R(\mathfrak D))'$ and either  $p_0$ or $q_0$ is $0$ by \cite[Theorem 6.2.7]{kad}.

 If $q_0=0$,  then we  have $q\leq p$. If $p_0=0$, then we  obtain that $p\leq q$. Next we show that $p=q$. Note that $\mathfrak N=[R(q)\mathfrak A]_2=[R(p)\mathfrak A]_2\oplus^{col}[(R(q-p))\mathfrak A]_2=H^2\oplus^{col}[(R(q-p))\mathfrak A]_2$. It trivial that $L^2(\mathcal M)=[R(p)\mathcal M]_2$ is orthogonal with $[(R(q-p))\mathcal M]_2$. Then we obtain that $[(R(q-p))\mathcal M]_2=0$ and  $p=q$. Consequently, we have $q\leq p$.
\end{proof}

\begin{thm}
Let $\mathfrak A$ be a type 1 subdiagonal algebra  and  $\mathfrak B$  a $\sigma$-weakly closed subalgebra containing $\mathfrak A$  of $\mathcal M$. Then the following assertions hold:

$(1)$ There exists a projection $E\in Lat_{\mathcal M}\mathfrak A \cap \mathfrak D$ such that $\mathfrak B=E\mathcal M + (I-E)\mathfrak A$.

$(2)$ $\mathfrak B$ is a type 1 subdiagonal algebra with respect to the unique faithful normal expectation $\Psi$ from $\mathcal M$ onto the diagonal $\mathfrak B\cap\mathfrak B^*$ of $\mathfrak B$ such that $\varphi\circ\Psi=\varphi$.
\end{thm}
\begin{proof}
$(1)$ As in the proof of Lemma 4.1, there exists a family of column orthogonal  partial isometries $\mathcal W$ and a projection $E $ in $\mathcal M$ such that $\mathfrak B=\oplus_{n\geq 1}^{col}W_n \mathfrak A\oplus ^{col}E \mathcal M $ and $[\mathfrak B h_{0}^{\frac12}]_2=\oplus_{n\geq 1}^{col}W_n H^2\oplus ^{col}E L^2(\mathcal M) $. Put $\mathfrak N=\oplus_{n\geq 1}^{col}W_n H^2\subseteq[\mathfrak B h_{0}^{\frac12}]_2$ be the type 1 part of $[\mathfrak Bh_0^{\frac12}]_2$. Let $p$ and $q$ as in Lemma 4.1. We   have $q\leq p$. This means that  $\mathfrak N =[R(q)\mathfrak A]_2\subseteq [R(p)\mathfrak A]_2=H^2$.   Thus we have $W_n \in \mathfrak A$ since $\mathfrak N=\oplus_{n\geq 1}^{col}W_n H^2\subseteq H^2$. Since $W_n^*\mathfrak B=W_n^*W_n\mathfrak A\subseteq \mathfrak A$, we get $W_n^* \in \mathfrak A$. Therefore $W_n, W_n^*\in \mathfrak D$  and $E\in \mathfrak D\subseteq \mathfrak A$ from the fact that $\oplus_{n\geq 1}^{col}W_nW_{n}^*=I-E\subseteq \mathfrak D$.
   Note that $E\in Lat_{\mathcal M}\mathfrak B\subseteq Lat_{\mathcal M}\mathfrak A $, $E\in \mathfrak D$ and   $\mathfrak B=\oplus_{n\geq 1}^{col}W_n \mathfrak A\oplus ^{col}E \mathcal M $. We have $(I-E)\mathfrak B=(I-E)(\oplus_{n\geq 1}^{col}W_n \mathfrak A)=\oplus_{n\geq 1}^{col}W_n \mathfrak A\subseteq (I-E)\mathfrak A \subseteq(I-E)\mathfrak B$ and so $ \oplus_{n\geq 1}^{col}W_n \mathfrak A = (I-E)\mathfrak A$.    Thus, $\mathfrak B=E\mathcal M + (I-E)\mathfrak A$.

 $(2)$ By (1), $\tilde{\mathfrak D}=\mathfrak B\cap \mathfrak B^*=E\mathcal ME + (I-E)\mathfrak D$. We next claim that $\mathfrak B$ is $\sigma_t^\varphi$-invariant. In fact, $E$ is in the center of $\mathfrak D$. Thus we have $\sigma_t^\varphi(E)=E$ for all $t\in\mathbb R$. It follows that  $\sigma_t^\varphi(\mathfrak B)= \mathfrak B$  and $\sigma_t^\varphi(\tilde{\mathfrak D})=\tilde{\mathfrak D}$  for all $t\in\mathbb R$.  By \cite[Chapter IX, Theorem 4.2]{tak}, there exists   unique  faithful normal conditional expectation $\Psi$ from $\mathcal M$ onto  $\tilde{\mathfrak D}$ such  that $\varphi\circ\Psi=\varphi$. We show that $\mathfrak B$ is a maximal subdiagonal algebra of $\mathcal M$ with respect to $\Psi$. In fact, $\mathfrak B_0=\{B\in\mathfrak B:\Psi(B)=0\}=E\mathcal M(I-E)+(I-E)\mathfrak A_0.$ An elementary calculation shows that $\mathfrak B_0$ is an ideal of $\mathfrak B$. We thus  have  that  $\Psi$ is multiplicative on $\mathfrak B$. It is maximal subdiagonal by \cite[Theorem 1.1]{xu}.

Finally we show that $\mathfrak B$ is of type 1. Note that $E$ is in the center of $\mathfrak D$, we have $\Phi(E  T(I-E))=0$ for all $T\in\mathcal M$. Moreover, $E(\mathfrak A+\mathfrak A^*)(I-E)=E\mathfrak A(I-E)\subseteq \mathfrak A$. We have $E\mathcal M(I-E)\subseteq \mathfrak A_0$. It follows that $\mathfrak B_0\subseteq \mathfrak A_0$. Consequently,
 $\cap_{n\geq 1}[h_{0}^{\frac12}\mathfrak B_0^n]_2\subseteq \cap_{n\geq 1}[h_{0}^{\frac12}\mathfrak A_0^n]_2=\{0\}$.
Therefore, $\mathfrak B$ is a type 1 subdiagonal algebra from \cite[Proposition 2.2]{jig2}.
\end{proof}
\section{The relative invariant subspace lattice of a type 1 subdiagonal algebra }
We recall that the  invariant subspace lattice of a subalgebra $\mathcal A\subseteq \mathcal \mathcal B(\mathcal H)$ is the set
$Lat\mathcal A=\{E\in\mathcal B(\mathcal H)_p:(I-E)AE=0, \forall A\in\mathcal A\}$.  If $\mathcal A\subseteq \mathcal M$, then
$Lat_{\mathcal M}\mathcal A=\{E\in\mathcal M_p: (I-E)AE=0, \forall A\in\mathcal A\}=Lat\mathcal A\cap\mathcal M_p$ is the relative invariant subspace lattice of $\mathcal A$ in $\mathcal M$.
  A subspace lattice is a nest if it is linearly ordered by inclusion.   Gilfeather and Larson in \cite{gil0} studied  the  relative invariant subspace lattices of  certain subalgebras in a von Neumann algebra
  and asked: If $\mathcal A$ is a
subalgebra in a factor von Neumann algebra $\mathcal M$ such that $\mathcal A+\mathcal A^*$ is dense in $\mathcal M$ in some topology for which
multiplication is separately continuous, is  $Lat_{\mathcal M}\mathcal A$ a
nest? Although this  question was negatively answered for the case in which $\mathcal M=\mathcal B(\mathcal H)$ in \cite{an},
 it is still interesting to consider those subalgebras with additional
properties. For example, how is the relative invariant subspace lattice of a maximal subdiagonal algebra?
 In \cite{jig0}, we shown that  the relative invariant subspace lattice of a  finite subdiagonal algebra is commutative.
 As an application of Theorem 4.2, we prove that  the relative invariant subspace lattice $Lat_{\mathcal M}\mathfrak A$ of a type 1 subdiagonal algebra  $\mathfrak A$ in $\mathcal M$ is commutative.

   We note that the centralizer $\mathcal M^\varphi$ of von Neumann algebra $\mathcal M$ associated with $\varphi$ is defined to be the set $\mathcal M^\varphi=\{A\in \mathcal M: \varphi(AB)=\varphi(BA),\forall B\in \mathcal M\}$(cf.\cite{kad,tak}). In fact,  $\mathcal M^\varphi$ is just  the fixed point algebra of $\mathcal M$ with respect to $\{\sigma_t^\varphi\}_{t\in \mathbb{R}}$: $\mathcal M^\varphi=\{A\in \mathcal M: \sigma_t^\varphi(A)=A,\forall t\in \mathbb R\}$. We remark that $\mathcal M^\varphi$ is a finite von Neumann algebra.

\begin{thm}
Let $\mathfrak A$ be a type 1 subdiagonal algebra of $\mathcal M$. Then the relative invariant subspace lattice $Lat_{\mathcal M}\mathfrak A$ of $\mathfrak A$ in $\mathcal M$ is commutative.
\end{thm}
\begin{proof}
Take any  $Q\in Lat_{\mathcal M}\mathfrak A$ and let  $\mathfrak B$ be the $\sigma$-weakly closed subalgebra generated by $Q$ and $\mathfrak A$.   Then $\mathfrak B$ is   a $\sigma_t^{\varphi}$-invariant type 1 subdiagonal algebra by Theorem 4.2.
Now  $\mathfrak B$ can be expressed as $\mathfrak B=\vee\{Q\mathfrak AQ\}+\vee\{Q\mathfrak A(I-Q)\}+\vee\{(I-Q)\mathfrak A(I-Q)\}$. Since $Q\in Lat_{\mathcal M}\mathfrak A$,  $Q\mathfrak A^*(I-Q)=0$. It follows that $Q\mathcal M(I-Q)=\vee\{Q(\mathfrak A+\mathfrak A^*)(I-Q)\}=\vee\{Q\mathfrak A(I-Q)\}$. Thus we have $Q\mathcal M(I-Q)\subseteq \mathfrak B_0$ and   $\Psi(QA(I-Q))=0$ for all $A\in\mathcal M$. This implies that  $\varphi(QA(I-Q))=\varphi\circ\Psi(QA(I-Q))=0$ for all $A\in\mathcal M$. That is $Q\mathcal M( I-Q)\subseteq\ker \varphi$. Since $\varphi$ is a state,  we also have that
  $(I-Q)\mathcal MQ\subseteq \ker \varphi$.
   For every $T\in \mathcal M$, $\varphi(TQ)=\varphi(QTQ)+\varphi((I-Q)TQ)$ and $\varphi(QT)=\varphi(QTQ)+\varphi(QT(I-Q))$. Then we have $\varphi(TQ)=\varphi(QT)$. It implies that $Q \in \mathcal M^\varphi$. Therefore, $Lat_{\mathcal M}\mathfrak A\subseteq\mathcal M^\varphi$ and we in turn obtain that $Lat_{\mathcal M}\mathfrak A \subseteq Lat_{\mathcal M^\varphi}(\mathfrak A\cap\mathcal M^\varphi)$. Note that $\mathcal M^\varphi$ is a finite von Neumann algebra and $\mathfrak A\cap\mathcal M^\varphi$ is a finite subdiagonal  algebra by \cite[Corollary 2.5]{jio}. Thus $Lat_{\mathcal M^\varphi}(\mathfrak A\cap\mathcal M^\varphi)$ is commutative by \cite[Theorem 2.1]{jig0}.  It is trivial that    $Lat_{\mathcal M}\mathfrak A$ is  also commutative.
\end{proof}

\begin{cor} Let $\mathfrak A$ be a type 1 subdiagonal algebra of $\mathcal M$. Then for any $Q\in Lat_{\mathcal M}\mathfrak A$, $Q\mathcal M(I-Q)\subseteq \mathfrak A_0$.
\end{cor}
\begin{proof} Note that $Q\in\mathcal M^{\varphi}$ and $Q\mathcal M(I-Q))\subseteq \ker \varphi$. Thus
 $\mbox{tr}(Th_0Q)=\mbox{tr}(QTh_0)=\varphi(QT)=\varphi(TQ)=\mbox{tr}(TQh_0)$ for all $T\in\mathcal M$. That is,
 $\mbox{tr}(T(h_0Q-Qh_0))=0$ for all $T\in\mathcal M$. Thus $Qh_0=h_0Q$. Furthermore, $\mbox{tr}(Ah_0QT(I-Q))=\mbox{tr}(AQh_0T(I-Q))=\mbox{tr}((I-Q)AQh_0T)=0$ for all $A\in \mathfrak A$. It follows that $QT(I-Q)\in (H^1)^{\bot}=\mathfrak A_0$ for all $T\in\mathcal M$.
\end{proof}
Similar to \cite[Theorem 2.4]{jig0}, we  next  show that  $Lat_{\mathcal M}\mathfrak A$  is a complete lattice generated by a nest in $Lat_{\mathcal M}\mathfrak A$ together with the lattice of projections in the center of $\mathcal M$.

\begin{thm}
Let $\mathfrak A$ be a type 1 subdiagonal algebra of $\mathcal M$. Then the relative invariant subspace lattice $Lat_{\mathcal M}\mathfrak A$ of $\mathfrak A$ in $\mathcal M$ is the complete lattice generated by a nest in $Lat_{\mathcal M}\mathfrak A$ together with the lattice of projections in the center of $\mathcal M$.
\end{thm}
\begin{proof}It is known that $Lat_{\mathcal M}\mathfrak A$ is commutative by Theorem 5.1.
 Take any  $E,F\in Lat_{\mathcal M}\mathfrak A$.   Then $E\mathcal M(I-E)\subseteq \mathfrak A_0$  by Corollary 5.2. Thus we get $(I-F) E\mathcal M(I-E)F\subseteq (I-F)\mathfrak AF=\{0\}$ and we in turn have $C_{E(I-F)}C_{(I-E) F}=0$, where $C_P$ denotes the central carrier of a projection $P$.

Note that $Lat_{\mathcal M}\mathfrak A$ contains the lattice of projections in the center of $\mathcal M$. We next show that $Lat_{\mathcal M}\mathfrak A$ satisfies the condition of \cite[Theorem 4.2]{gil}. In fact, for $E,F\in Lat_{\mathcal M}\mathfrak A$, we have $C_{E(I-F)}C_{(I-E) F}=0$, then $E(I-F)\leq C_{E(I-F)}$ while $(I-E) F\leq I-  C_{E(I-F)} $. It follows that the reflexive operator algebra
$$AlgLat_{\mathcal M}\mathfrak A=\{T\in\mathcal B(\mathcal H): (I-P)TP=0,\forall P\in Lat_{\mathcal M}\mathfrak A\}$$ with commutative subspaces lattice $Lat_{\mathcal M}\mathfrak A$ is a nest subalgebra of a von Neumann algebra $\mathfrak R$  with a nest $\mathcal N$ in $Lat_{\mathcal M}\mathfrak A$ by \cite[Theorem 4.2]{gil}. Thus $AlgLat_{\mathcal M}\mathfrak A=Alg\mathcal N\cap \mathfrak R$.
Note that $\mathfrak R\supseteq AlgLat_{\mathcal M}\mathfrak A\supseteq AlgLat_{\mathcal M}\mathfrak A\cap \mathcal M\supseteq \mathfrak A$. Then $\mathfrak R\supseteq [\mathfrak A+\mathfrak A^*]_{\infty}=\mathcal M$.
  Therefore, $AlgLat_{\mathcal M}\mathfrak A \cap \mathcal M=Alg\mathcal N\cap \mathfrak R\cap \mathcal M=Alg\mathcal N\cap \mathcal M$.   Similar to  \cite[Theorem 2.4]{jig0}, we have the desired result.
\end{proof}
\begin{cor}
Let $\mathfrak A$ be a type 1 subdiagonal algebra of $\mathcal M$. If $\mathcal M$ is a factor, then the relative invariant subspace  lattice $Lat_{\mathcal M}\mathfrak A$ of $\mathfrak A$ is a nest.
\end{cor}

\end{document}